\documentclass[12pt]{article}
\usepackage[margin=1in]{geometry}

\usepackage{amsmath}
\usepackage{amssymb}
\usepackage{amsthm}
\usepackage{verbatim}
\usepackage[pdftex]{hyperref}
\usepackage{graphicx}
\usepackage{wrapfig}
\usepackage{float}
\usepackage{tikz}
\usepackage{tikz-cd} 
\usepackage{quiver}
\usepackage[font=footnotesize]{caption}
\usepackage{subcaption}
\usepackage{hyperref}

\usepackage[T1]{fontenc}
\usepackage{etoolbox} 

\graphicspath{ {./images/} } 

\usepackage[T1]{fontenc}
\usepackage[utf8]{inputenc}

\newtheorem{theorem}{Theorem}[section]
\newtheorem*{theorem*}{Theorem}
\newtheorem{lemma}[theorem]{Lemma}

\theoremstyle{definition}
\newtheorem{definition}{Definition}[section]

\theoremstyle{remark}
\newtheorem{remark}[theorem]{Remark}

\newenvironment{packed_item}{
\begin{itemize}
\setlength{\itemsep}{1pt}
\setlength{\parskip}{0pt}
\setlength{\parsep}{0pt}
}{\end{itemize}}

\theoremstyle{definition}

\DeclareMathOperator{\im}{im}
\DeclareMathOperator{\Sym}{Sym}
\DeclareMathOperator{\symGE}{symGE}

\begin{document}

\title{On the Gauss-Epple homomorphism of the braid group $B_n$, and generalizations to Artin groups of crystallographic type}
\author{Joshua Guo\footnote{Newton South High School}
~ Kevin Chang\footnote{Department of Mathematics, Columbia University, New York, NY 10027} ~
}
\date{\today}
\maketitle

\vspace{-20pt}
\begin{abstract}
\normalsize
In this paper, we introduce a broad family of group homomorphisms that we name the Gauss-Epple homomorphisms. 
In the setting of braid groups, the Gauss-Epple invariant was originally defined by Epple based on a note of Gauss as an action of the braid group $B_n$ on the set $\{1, \dots, n\}\times\mathbb{Z}$; we prove that it is well-defined. 
We consider the associated group homomorphism from $B_n$ to the symmetric group $\Sym(\{1, \dots, n\}\times\mathbb{Z})$. 
We prove that this homomorphism factors through $\mathbb{Z}^n\rtimes S_n$ (in fact, its image is an order 2 subgroup of the previous group).
We also describe the kernel of the homomorphism and calculate the asymptotic probability that it contains a random braid of a given length. 
Furthermore, we discuss the super-Gauss-Epple homomorphism, a homomorphism which extends the generalization of the Gauss-Epple homomorphism and describe a related 1-cocycle of the symmetric group $S_n$ on the set of antisymmetric $n\times n$ matrices over the integers. 
We then generalize the super-Gauss-Epple homomorphism and the associated 1-cocycle to Artin groups of finite type. 
For future work, we suggest studying possible generalizations to complex reflection groups and computing the vector spaces of Gauss-Epple analogues.
\end{abstract}

\section{Introduction}

\subsection{Braid groups}

\begin{wrapfigure}[6]{r}{0.2\textwidth}
    \vspace{-20pt}
    \includegraphics[height=0.18\textwidth]{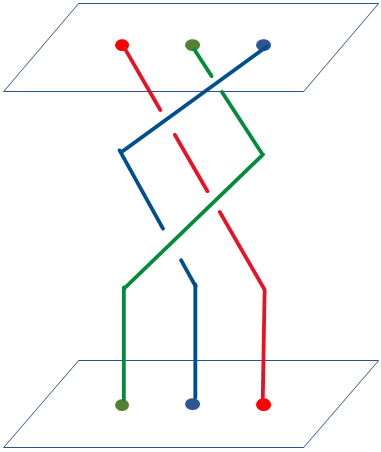}
    \caption{An example braid}
    \label{fig-example-braid}
\end{wrapfigure}

\hspace{\parindent}Fix, once and for all, an arbitrary integer $n\geq 2$; we shall use the notation $[n] := \{1, \dots, n\}$.
A \textit{braid} on $n$ strands is a topological object consisting of $n$ strands in 3-dimensional space whose endpoints are fixed to two distinguished parallel planes, such as in Figure \ref{fig-example-braid}.
(We assume, as is common, that the endpoints on each plane are collinear.
We also choose one plane to be the ``top'' and number the strands from $1$ to $n$ in their order on this plane; we choose the other to be the ``bottom''.)

\begin{figure}[h]
    \centering
    \includegraphics[width=0.4\textwidth]{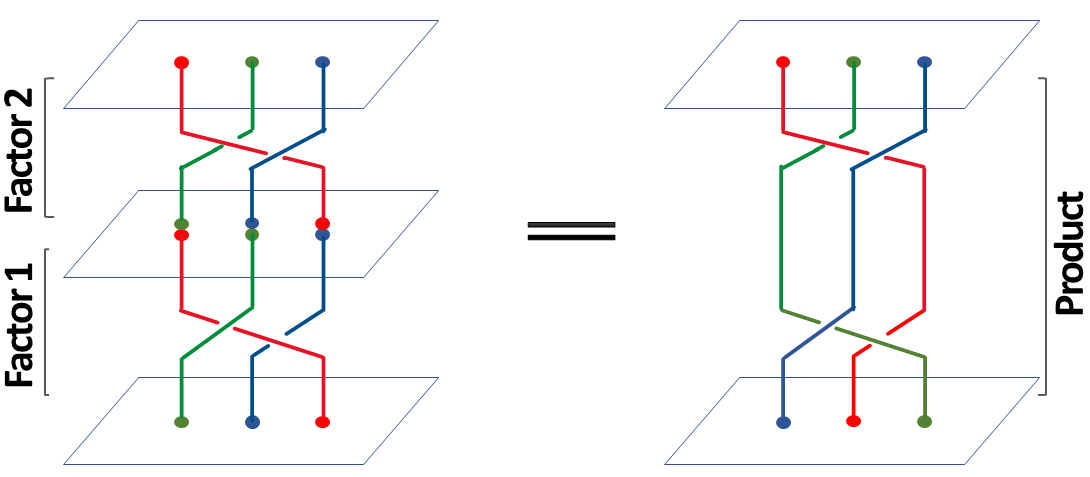} 
    \caption{Braid multiplication}
    \label{fig-braid-multiplication}
\end{figure}

By ``gluing'' two braids on $n$ strands together (i.e., identifying the ``bottom'' plane of the first with the ``top'' plane of the second, as depicted in Figure \ref{fig-braid-multiplication}), we can create a third braid on $n$ strands.
This operation, which we view as composition, endows the set of all braids on $n$ strands up to isotopy (i.e., topological deformation) with the structure of a monoid (i.e., a set with an associative and identity).
In fact, this operation is invertible; an example braid inverse is shown in Figure \ref{fig-braid-inverse}.
As a result, the set of all braids on $n$ strands up to isotopy is in fact a group, the \textit{braid group} $B_n$.

\begin{figure}[h]
    \centering
    \includegraphics[width=0.55\textwidth]{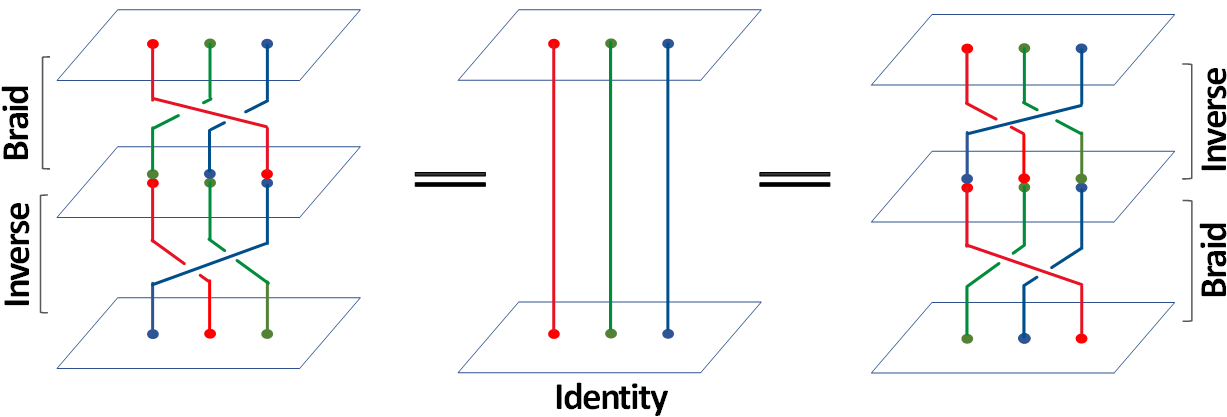} 
    \caption{A braid and its inverse}
    \label{fig-braid-inverse}
\end{figure}

The braid group $B_n$ is a well-studied mathematical object.
Emil Artin first proposed it in 1925 \cite{Artin1925} and further discussed it in a 1947 paper \cite{10.2307/1969218}.
It is well-known that $B_n$ has no torsion and that the full twist generates its center. Several efficient algorithms to solve $B_n$'s word and conjugacy problems are known \cite{gonzalez2011basic}.
It is also well-known (as the Nielsen-Thurston classification) that all braids are either \textit{periodic}, \textit{reducible}, or \textit{pseudo-Asonov} \cite{gonzalez2011basic}.
Furthermore, $B_n$ is linear, as there is a faithful (!) representation (the \textit{Lawrence-Krammer representation}) $B_n\to GL_{n(n-1)/2}(\mathbb{Z}[q^{\pm 1}, t^{\pm 1}])$ \cite{krammer2002braid}.

The braid group $B_n$ has a canonical generating set of $n-1$ generators known as the \textit{Artin generators} and conventionally denoted $\sigma_1, \dots, \sigma_{n-1}$.
The Artin generator $\sigma_i$ consists of the braid that twists the $i$th leftmost strand over and to the right of the $(i+1)$th strand.
With this generating set, the braid group has the following presentation:
\begin{align*}
    B_n = \langle \sigma_1, \dots, \sigma_{n-1} | \sigma_i\sigma_{i+1}\sigma_i = \sigma_{i+1}\sigma_i\sigma_{i+1}\forall i, \sigma_i\sigma_j = \sigma_j\sigma_i \forall i, j: |i - j| > 1\rangle.
\end{align*}

The \textit{writhe} is a group homomorphism $B_n\to\mathbb{Z}$ defined by the relation $\sigma_i\mapsto 1$. The image of any braid $\beta$ under writhe is called the \textit{writhe} of the braid, denoted $|\beta|$. The \textit{permutation} is a group homomorphism $B_n\to S_n$ defined by the relation $\sigma_i\mapsto(i, i + 1)$, where the right hand side is a transposition. The permutation of a braid is defined similarly and denoted $\pi_{\beta}$.

\begin{figure}[h]
    \centering
    \includegraphics[width=0.15\textwidth]{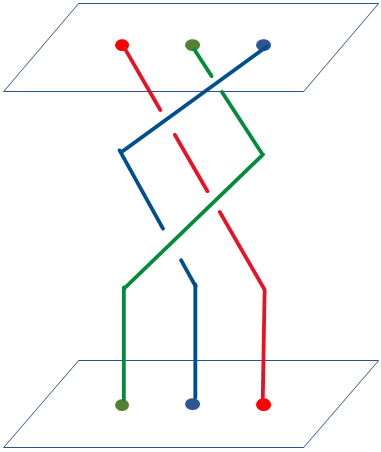}
    \caption{The braid $(\sigma_2^{-1}\sigma_1^{-1})^2$, which has permutation (1,3,2) and writhe~$-4$.}
\end{figure}

Both the writhe and permutation are surjective homomorphisms but \underline{not} injective.
The \textit{writhe-permutation homomorphism} $WP: B_n\to\mathbb{Z}\times S_n$ is defined by $\beta\mapsto (|\beta|, \pi_{\beta})$. It is neither surjective nor injective, as we shall prove later.

\subsection{Artin groups of crystallographic type}

\hspace{\parindent}An \textit{Artin group} $\mathcal{A}$ is a group presented by a finite set of generators and at most one braid relation (i.e., a relation of the form $a = b$, $ab = ba$, $aba = bab$, $abab = baba$, etc.) between any two generators.
The presentation of an Artin group can be depicted in a \textit{Dynkin diagram}, as shown in Figure \ref{fig-dynkin-diagrams}.

\begin{figure}[h]
    \vspace{10pt}
    \begin{center}
    \includegraphics[height=0.5\textwidth]{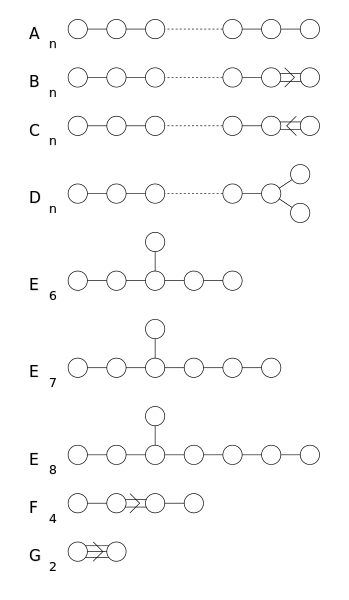}
    \caption{Dynkin diagrams of finite Coxeter groups
    \cite{dynkindiagrams2007commons}}
    \label{fig-dynkin-diagrams}
    \end{center}
\end{figure}
\begin{figure}
    \centering
    \includegraphics[height=100pt]{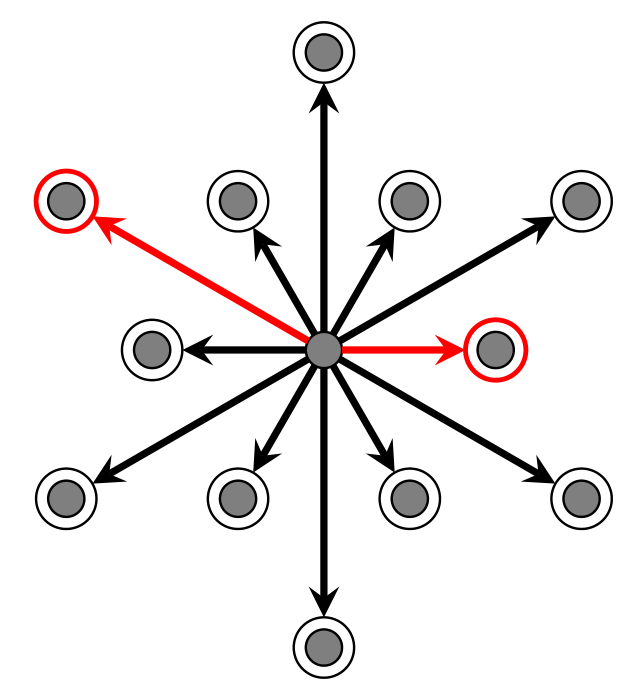}
    \caption{The $G_2$ root system with two simple roots highlighted \cite{nasmith2020medium}}
    \label{fig-g2-simpleroots}
\end{figure}
\vspace{-10pt}
A \textit{Coxeter group} $\mathcal{C}$ is a group presented with the generators and relations of an Artin group and the relations that the square of any generator is the identity.
A Coxeter group's presentation can also be depicted in a Dynkin diagram.

Sometimes, a Coxeter group is generated by the reflections upon a \textit{root system} $\Phi$, which is a special set of vectors.
(In this case, they are known as \textit{Weyl groups}.)
The associated Artin group $\mathcal{A}$ is then called an \textit{Artin group of crystallographic type}.
Furthermore, in the cases where the Coxeter group is indeed associated with a root system, an isomorphism between the Coxeter group and the group of linear transformations sends the canonical Coxeter group generators to the reflections associated with some choice of \textit{simple roots} $\Delta\subset\Phi$.
(We refer to the simple root associated to a canonical Artin generator $a$ of $\mathcal{A}$ as $\Delta_a$.)
One such set of simple roots, corresponding to the case commonly referred to as $G_2$, is shown in Figure \ref{fig-g2-simpleroots}.

\subsection{Complex reflection groups}

\begin{wrapfigure}[22]{r}{0.3\textwidth} 
    \vspace{-40pt}
        \includegraphics[height=0.7\textwidth]{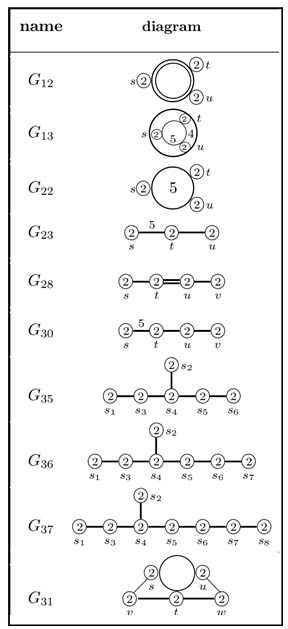} 
    \caption{Partial table of complex reflection groups \cite{Broue2010IntroductionTC}}
\end{wrapfigure}
\hspace{\parindent}We adopt the terminology and notation of Broué \cite{Broue2000}.

Let $V$ be some finite-dimensional $\mathbb{C}$-vector space.
A \textit{pseudo-reflection} is a linear transformation that acts trivially on some hyperplane (called its \textit{reflecting hyperplane}).

A \textit{complex reflection group} $W$ is a finite subgroup of $GL(V)$ generated by pseudo-reflections.

We denote the set of pseudo-reflections in $W$ by $\Psi$. (This notation is an exception; it comes from Dr. Trinh's notes \cite{mqt200207notes}, not from Broué.)
The hyperplane arrangement in $V$ consisting of the reflecting hyperplanes of the $\Psi$ is denoted by $\mathcal{A}$, and its hyperplane complement $V - \cup_{H\in\mathcal{A}} H$ is denoted as $M$.
We define the analogous \textit{braid group} $B := \pi_1(M/W)$; in other words, the fundamental group of $M$ up to rotation by elements of $W$.

Similar to the previous case, the braid groups $B$ have canonical Artin-like presentations, and the corresponding Coxeter groups $W$ have canonical Coxeter-like presentations \cite{Broue2000}.
More precisely, there exists a subset $\mathbf{S} = \{\mathbf{s}_1, \dots, \mathbf{s}_n\}$ of $B$ consisting of distinguished braid reflections, and a set $R$ of relations of the form $w_1 = w_2$, where $w_1$ and $w_2$
are positive words of equal length in the elements of $\mathbf{S}$, such that $\langle S|R\rangle$ is
a presentation of $B$.
Moreover, their images $s_1, \dots, s_n$ in $W$ generate $W$, and the group $W$ is presented by
$$\langle S|R;(\forall s\in S)(s^{e_s} =1)\rangle$$ where $e_s$ denotes the order of $s$ in $W$.

There is a quotient map $B\to W$; we denote the image of an element $b\in B$ under this quotient map simply as $W(b)$.

\subsection{Overview of results}

\hspace{\parindent}We study a large family of mathematical objects that we call the \textit{Gauss-Epple homomorphisms}.
This research helps us understand the structure of braid groups $B_n$, which describe the structure of \textit{braids}, a kind of topological object.

The first such homomorphism was implicitly introduced by Epple \cite{Epple98} as an action of $B_n$ based on a note by Gauss.
Around the time of his note, Gauss was primarily interested in topology for its applications to electromagnetism and celestial dynamics. In this work, we generalize this concept to a broader family of homomorphisms from Artin groups of finite type, a large family of groups including the braid groups.

Firstly, in Section \ref{sec-theGEhomor}, we prove that there exists a well-defined and unique left group action of $B_n$ on $\mathbb{Z}\times[n]$, as implicit in \cite{Epple98}.
We refer to this action as the \textit{Gauss-Epple action}.
It is equivalent to a group homomorphism from $B_n$ to $\Sym(\mathbb{Z}\times[n])$, which we call the \textit{Gauss-Epple homomorphism} (denoted $GE$).
In Subsection \ref{subsec-imWP}, we describe the \textit{writhe-permutation} homomorphism $WP: B_n\to\mathbb{Z}\times S_n$ (the homomorphism that maps a braid to its writhe and underlying permutation) in greater detail, as we treat this homomorphism as a toy model of $GE$. We show that the image of \textit{writhe-permutation} homomorphism $WP: B_n\to\mathbb{Z}\times S_n$ is a particular order 2 subgroup of $\mathbb{Z}\times S_n$.
Then, in Subsection \ref{subsec-imGE}, we show that the image of $GE$ is (isomorphic to) an order 2 subgroup of $\mathbb{Z}^n\rtimes S_n$.
In Subsection \ref{subsec-kerGE}, we discuss the kernel of $GE$. We find that this group is strictly contained in the kernel of the \textit{writhe-permutation homomorphism}.

We summarize these results with the following commutative diagram:

\[\begin{tikzcd}
	{B_ n} & {\mathbb{Z}^n\rtimes S_n} & {\Sym([n]\times\mathbb{Z})} \\
	& {\mathbb{Z}\times S_n}
	\arrow[from=1-1, to=1-2]
	\arrow[from=1-2, to=1-3]
	\arrow[from=1-2, to=2-2]
	\arrow["WP", from=1-1, to=2-2]
	\arrow["GE", curve={height=-12pt}, from=1-1, to=1-3]
\end{tikzcd}\]

Each of the maps in this diagram is a group homomorphism, as explained below:

\begin{itemize}
    \item The map $GE$, from $B_n$ to $\mathrm{Sym}(\mathbb{Z}^2)$, is the Gauss-Epple action of $B_n$ on $\mathbb{Z}^2$.
    \item The map from $B_n$ to $\mathbb{Z}^n\rtimes S_n$ maps a braid to the tuple of its vector and its permutation.
    \item The map from $\mathbb{Z}^n\rtimes S_n$ to $\mathrm{Sym}(\mathbb{Z}^2)$ is the map $(\pi, \ell)\mapsto ((a, b) \mapsto (\pi(a), b + \ell_a))$.
    \item The map $WP$, from $B_n$ to $\mathbb{Z}\times S_n$, maps a braid to its writhe and braid permutation.
    \item The map from $\mathbb{Z}^n\rtimes S_n$ to $\mathbb{Z}\times S_n$ maps $(\ell, \pi)$ to $(\sum\ell, \pi)$.
\end{itemize}

In Section \ref{sec-symGE}, we briefly discuss another action of $B_n$ (the \textit{symmetric-Gauss-Epple action}) mentioned by Epple.
It turns out that this action has very similar properties to the Gauss-Epple action, including sharing the same kernel.

In Section \ref{sec-superGE}, we discuss the \textit{super-Gauss-Epple homomorphism} (denoted $SGE$), a homomorphism of $B_n$ that refines $GE$.
To describe the image of $SGE$, we introduce a 1-cocycle of the symmetric group $S_n$ on the set of $n\times n$ antisymmetric matrices, which we prove has a remarkable nonnegativity property.

Finally, in Section \ref{sec-ATgroupsWeyl}, we generalize our results to the contexts of Artin groups of finite type, a broad family of groups that generalize the braid groups.
More specifically, we introduce a family of novel homomorphisms (with domain $\mathcal{A}$ and range $\mathbb{Z}^{\Phi}\rtimes\mathcal{C}$) that are analogous to the super-Gauss-Epple homomorphism in the classical case of $B_n$.
We also note corresponding 1-cocycles, which also share an analogous nonnegativity property.

We close in Section \ref{sec-complex} with remarks about complex reflection groups.

\subsubsection{Acknowledgements}

\hspace{\parindent}We thank our advisor, Minh-Tâm Trinh, for proposing this project and giving us helpful directions.
We also thank the MIT PRIMES program for supporting our research.

\section{The Gauss-Epple homomorphism}\label{sec-theGEhomor}

\begin{figure}
    \begin{center}\includegraphics[width=0.5\textwidth]{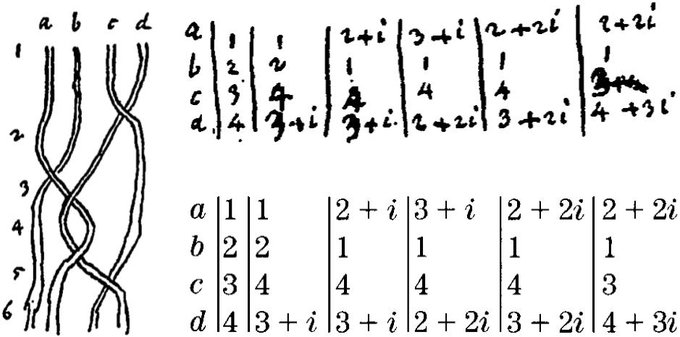}\end{center}
    \caption{Page 283 of Gauss's Handbuch 7 \cite{Epple98}}
\end{figure}

\hspace{\parindent}We formally introduce the Gauss-Epple action. 
This action was initially defined by Epple \cite{Epple98} based on a note of Gauss.
Since Gauss's notation involved complex numbers (more precisely, the Gaussian integers $\mathbb{Z}[i]$), the Gauss-Epple action was originally defined in terms of complex numbers.
For notational and theoretical simplicity, however, we will define the Gauss-Epple action as an action of $B_n$ on $[n]\times\mathbb{Z}$ instead of on $\mathbb{Z}[i]$.
Our reasons for doing so should become clear shortly.
We also view the Gauss-Epple action as a group homomorphism from $B_n$ to $\Sym([n]\times\mathbb{Z})$.
(This is a specific instance of .the classical trick to convert between group actions $G\times X \to X$ and group homomorphisms $G \to (X \to X)$: by currying the input.)

As defined by Epple (based on Gauss's notation), the Gauss-Epple homomorphism is given as follows:
\begin{definition}[Gauss-Epple action]\label{def-eppleGE}
For any $n\in\mathbb{N}$, the Gauss-Epple action $\alpha: B_n\times \mathbb{Z}[i]\to \mathbb{Z}[i]$ is the unique left group action of $B_n$ (with canonical Artin generators $\sigma_1, \sigma_2, \dots$) on the Gaussian integers $\mathbb{Z}[i]$ defined by the following generating relation:

$$\alpha(\sigma_k, z) := \begin{cases}
z &  \Re(z)\notin\{k, k + 1\}\\
z + 1 &  \Re(z) = k\\
z - 1 + i &  \Re(z) = k + 1
\end{cases}.$$
\end{definition}

We simplify this definition by replacing each complex number involved with the ordered pair of its real and complex parts, and then restricting the first component to the elements of $[n]$.
This notational change yields an action on the elements of $[n]\times\mathbb{Z}$.
Therefore, as we define it, the Gauss-Epple action is given as follows:

\begin{definition}[Gauss-Epple action]\label{def-GE}
For any $n\in\mathbb{N}$, the Gauss-Epple action $GE: B_n\times ([n]\times\mathbb{Z})\to [n]\times\mathbb{Z}$ is the unique left group action of $B_n$ (with canonical Artin generators $\sigma_1, \sigma_2, \dots$) on $[n]\times\mathbb{Z}$ defined by the following generating relation:

$$GE(\sigma_k, (a, b)) := \begin{cases}
(a, b) & a\notin\{k, k + 1\}\\
(k + 1, b) & a = k\\
(k, b + 1) & a = k + 1
\end{cases}.$$
\end{definition}

\begin{remark}
Note that this definition implies that the inverses of the Artin generators act as follows:

$$GE(\sigma_k^{-1}, (a, b)) = \begin{cases}
(a, b) & a\notin\{k, k + 1\}\\
(k + 1, b - 1) & a = k\\
(k, b) & a = k + 1
\end{cases}.$$
\end{remark}
\vspace{30pt}

We now prove that the Gauss-Epple action is well-defined, a fact stated without proof by Epple \cite{Epple98}.

\begin{lemma}\label{lem-GE-welldefined}
The Gauss-Epple action $GE$ is uniquely defined by Definition \ref{def-GE} as a left group action of $B_n$ on $\mathbb{Z}^2$.
\end{lemma}
\begin{proof}
Since the Artin generators satisfy the generating relations $\sigma_k \sigma_l = \sigma_l \sigma_k$ for all $k, l$ such that $|k - l| \geq 2$ and $\sigma_k\sigma_{k+1}\sigma_k = \sigma_{k+1}\sigma_k\sigma_{k+1}$ for all $k\in[n - 1]$, it is enough to verify that that the analogous relations hold for the Gauss-Epple action to prove Lemma \ref{lem-GE-welldefined}.

Suppose that $k, l$ are two integers satisfying $|k - l|\geq 2$. We observe that $\{k, k + 1\}\cap\{l, l + 1\}=\varnothing$, which will become important later. Let $a, b$ be arbitrary members of $\mathbb{Z}$.
We prove that, for all integers $k, l$ satisfying $|k - l|\geq 2$, we have $GE(\sigma_k)GE(\sigma_l) = GE(\sigma_l)GE(\sigma_k)$, as follows:

\begin{packed_item}
\item If $a\notin\{k, k + 1, l, l + 1\}$, then $\sigma_k\sigma_l(a, b) = (a, b) = \sigma_l\sigma_k(a, b)$.
\item If $a\in\{k, k +1\}$, then $\sigma_k\sigma_l(a, b) = \sigma_k(a, b) = \sigma_l\sigma_k(a, b)$. 
(From the second to third step, we use the fact that the first component of $\sigma_k(a, b)$ is in $\{k, k + 1\}$, and so outside of $\{l, l + 1\}$.)
\item If $a\in\{l, l + 1\}$, then $\sigma_k\sigma_l(a, b) = \sigma_l(a, b) = \sigma_l\sigma_k(a, b)$.
\end{packed_item}

Again, suppose that $k, a, b$ are arbitrary integers satisfying $k\in\{1, \dots, n - 2\}$.
Then we prove that we have $GE(\sigma_k)GE(\sigma_{k+1})GE(\sigma_k) = GE(\sigma_{k+1})GE(\sigma_k)GE(\sigma_{k+1})$ as follows:

\begin{packed_item}
\item If $a\notin\{k, k + 1, k + 2\}$, then $\sigma_k\sigma_{k+1}\sigma_k(a, b) = (a, b) = \sigma_{k+1}\sigma_k\sigma_{k+1}(a, b)$.
\item If $a = k$, then $\sigma_k\sigma_{k+1}\sigma_k(a, b) = \sigma_k\sigma_{k+1}(k + 1, b) = \sigma_k(k + 2, b) = (k + 2, b)$,
and $\sigma_{k+1}\sigma_k\sigma_{k+1}(a, b) = \sigma_{k + 1}\sigma_k(k, b) = \sigma_{k+1}(k + 1, b) = (k + 2, b)$.
\item If $a = k + 1$, then $\sigma_k\sigma_{k+1}\sigma_k(a, b) = \sigma_k\sigma_{k+1}(k, b + 1) = \sigma_k(k, b + 1) = (k + 1, b + 1)$,
and $\sigma_{k+1}\sigma_k\sigma_{k+1}(a, b) = \sigma_{k + 1}\sigma_k(k + 2, b) = \sigma_{k+1}(k + 2, b) = (k + 1, b + 1)$.
\item If $a = k + 2$, then $\sigma_k\sigma_{k+1}\sigma_k(a, b) = \sigma_k\sigma_{k+1}(k + 2, b) = \sigma_k(k + 1, b + 1) = (k, b + 2)$,
and $\sigma_{k+1}\sigma_k\sigma_{k+1}(a, b) = \sigma_{k + 1}\sigma_k(k + 1, b + 1) = \sigma_{k+1}(k, b + 2) = (k, b + 2)$.
\end{packed_item}
\end{proof}

From now on, we consider GE as a group homomorphism $GE: B_n\to\Sym([n]\times\mathbb{Z})$

\begin{remark}
Note that, by letting $n$ vary, we can extend the Gauss-Epple homomorphism $GE: B_n \to\Sym([n]\times\mathbb{Z})$ to a homomorphism $GE: B_{\infty}\to\Sym(\{1, 2, \dots\}\times\mathbb{Z})$, where the group $B_{\infty}$ is defined to be the direct limit of the groups $B_n$. 
However, we do not pursue this direction.
\end{remark}

\subsection{Image of WP}\label{subsec-imWP}
\hspace{\parindent}Firstly, note the following classical fact, which motivates the study of 
\text{images} of group homomorphisms:

\begin{theorem}[First Isomorphism Theorem]\label{thm-1st-isothm}
Let $\phi: G\to H$ be a group homomorphism. Then $\im(\phi) \cong G/\ker\phi$.
\end{theorem}

For the group homomorphisms we study, it turns out that the images are much simpler than the kernels. Hence we will completely determine the images but only briefly discuss the kernels.

Since the Gauss-Epple action permutes the $a$ part without respect to the $b$ part according to the permutation of the braid, it is evident that braids in the kernel of the Gauss-Epple action have the identity permutation (i.e., are pure).
Early on, however, we discovered that braids in the kernel of the Gauss-Epple homomorphism have zero writhe too (we will prove this in Lemma \ref{lem-sum-link-writhe}).
Hence we decided to study the invariant $WP$ (which stands for ``writhe-permutation''), a weaker invariant than the Gauss-Epple homomorphism.
This invariant, too, turns out to be a group homomorphism: in this case, with a range of $\mathbb{Z}\times S_n$.
In fact, we find that the image of $WP$ is an order 2 subgroup of the range:

\begin{theorem}[Structure of imWP]\label{thm-structure-imWP}
There is a braid on $n$ strands with writhe $w\in\mathbb{Z}$ and permutation $\pi\in S_n$ iff (if and only if) $\pi$ and $w$ have the same parity.
\end{theorem}

\begin{proof}
\textbf{Necessity}: Suppose that there is a braid on $n$ strands with writhe $w\in\mathbb{Z}$ and permutation $\pi\in S_n$. 
If the number of twists of the braid is even, both $\pi$ and $w$ will also necessarily be even.
Similarly, if the number of twists of the braid is odd, both $\pi$ and $w$ will necessarily be odd.
Thus the claim holds.

\textbf{Sufficiency}: Let $\pi$ and $w$ be arbitrary permutations and integers with the same parity.
Since $S_n$ is generated by $(1, 2), (2, 3), \dots, (n - 1, n)$, we can write $\pi = \prod_{i = 1}^f t_{a_i}$, where $t_i = (i, i + 1)$ and $a_i$ is some arbitrary sequence of integers.

Consider the braid $\beta' = \prod_{i = 1}^f \sigma_{a_i}$, where $a_i$ is as before. 
Trivially, $\beta'$ has permutation $\pi$.
By the argument in the previous part of the proof, the writhe of $\beta'$ has the same parity as $\pi$, which has the same parity as $w$, so that $\sigma_1^{w - |\beta'|}$ is pure.
We conclude that the braid $\beta = \beta'\sigma_1^{w - |\beta'|}$ has writhe $w$ and permutation $\pi$, as desired,  and we are done.
\end{proof}
It is clear that since the image of $WP$ is not $B_n$ (nor is it the image of $GE$), that $WP$ must have a nontrivial kernel, and moreover, an element lying outside of the kernel of $GE$.
We can make this explicit by giving the example $\sigma_1^{-2}\sigma_2^2$.

\subsection{Image of GE}\label{subsec-imGE}

\hspace{\parindent}The computation of the image of the Gauss-Epple homomorphism (henceforth abbreviated \textit{imGE}) is similar, albeit much more tedious. Firstly, we prove that the Gauss-Epple action factors through $\mathbb{Z}^n\rtimes S_n$.
(The group rule of $\mathbb{Z}^n\rtimes S_n$ is $(\pi, \ell)(\pi', \ell') = (\pi\pi', \ell' + \pi'\ell)$.)

\begin{theorem}[Structure theorem for imGE]\label{thm-factor-imGE}
Define the group inclusion $\iota: \mathbb{Z}^n\rtimes S_n\to\Sym([n]\times\mathbb{Z})$ by the equation $(\pi, \ell)\to ((a, b)\to (\pi(a), b + \ell(a))$,
and define the group surjection $\varsigma: B_n\to\mathbb{Z}^n\rtimes S_n$ by 
$$\varsigma(\sigma_k) := ((k, k + 1), \vec{e_k}).$$
Then $GE = \varsigma\circ\iota$.
\end{theorem}
\begin{proof}

It is trivial to verify that $GE = \varsigma\circ\iota$ for the Artin generators.
By induction, and by the fact that all three functions are group homomorphisms, it follows that $GE = \varsigma\circ\iota$ for all braids.
\end{proof}

As an aside, we observe the Gauss-Epple action is transitive when restricted to $[n]\times\mathbb{Z}$.
However, as we can deduce from Theorem \ref{thm-factor-imGE}, it is NOT doubly transitive.

We also relate imWP to imGE with the following fact (which motivates the earlier study of imWP):

\begin{lemma}\label{lem-sum-link-writhe}
Let $\beta$ be a braid, and suppose that $\varsigma(\beta) = (\pi, \ell)$ (where $\varsigma$ is defined as in Theorem \ref{thm-factor-imGE}). Then the sum of the components of $\ell$ is the writhe of $\beta$.
\end{lemma}
\begin{proof}

Note that $\varkappa: \mathbb{Z}^n\rtimes S_n\to\mathbb{Z}$ defined by $\varkappa((\pi, \ell)) := \sum_i \ell_i$ is a group homomorphism.
We know that $\varkappa$ is a group homomorphism since, for any $\pi, \ell, \pi', \ell'$, we have \begin{align*}
    \varkappa((\pi, \ell)(\pi', \ell')) & = \varkappa((\pi\pi', \ell\pi' + \ell'))\\
    &= \sum_i (\ell(\pi'(i)) + \ell'(i))\\
    &= \sum_i \ell(\pi'(i)) + \sum_i \ell'(i)\\
    &= \sum_i \ell(i) + \sum_i \ell'(i)\\
    &= \varkappa((\pi, \ell)) + \varkappa((\pi', \ell')),
\end{align*}
as desired.

Hence $\varkappa\circ\varsigma$ is a group homomorphism from $B_n$ to $\mathbb{Z}$.

We also note that writhe is a group homomorphism from $B_n$ to $\mathbb{Z}$.
Furthermore, it is easy to verify that $\varkappa\circ\varsigma(\sigma_k) = |\sigma_k| = 1$.
Since $\varkappa\circ\varsigma$ and writhe agree on the generators of $B_n$, they must agree everywhere by induction, and so be the same.
\end{proof}
To fully characterize imGE, we begin with the special case of pure braids.

\begin{lemma}\label{lem-all-evensum-linkvecs-possible}
For any $\ell\in\mathbb{Z}^n$, there is a pure braid $\beta$ with $\ell_{\beta} = \ell$ iff $\ell$ has even sum.
\end{lemma}
\begin{proof}
\textbf{Necessity}: Suppose that there exists such a pure braid $\beta$. Then $\sum_a\ell_a = |\beta|$, and the latter is even since $\beta$ is pure. Hence $\sum_a\ell_a$ must be pure, and we are done.

\textbf{Sufficiency}:
Let $\beta_1$ and $\beta_2$ be arbitrary pure braids. Then $\ell_{\beta_1\beta_2} = \ell_{\beta_1} + \ell_{\beta_2}$, and $\ell_{\beta_1^{-1}} = -\ell_{\beta_1}$.
Hence the set of all possible $\ell_{\beta}$ values for braids $\beta$ forms a lattice in $\mathbb{Z}^n$.

We compute that $\ell_{\sigma_k^2} = (0, \dots, 0, 1, 1, 0, \dots, 0)$, where the $1, 1$ is in the $k$th and $(k+1)$th places and where we have identified $\ell$ with a vector in $\mathbb{Z}^n$.
Similarly, we compute $\ell_{(\sigma_k\sigma_{k + 1})^3} = (0, \dots, 0, 2, 2, 2, 0, \dots, 0)$.

Since $(a, a + b - 2c, c) = (a - 2c)(1, 1, 0) + (b - 2c)(0, 1, 1) + c(2, 2, 2)$, we conclude that these vectors span the lattice of all elements of $\mathbb{Z}^n$ with even sum, and we are done.
\end{proof}

Then, we will combine this result with some further lemmas to characterize the image in Theorem \ref{thm-imGE}.
Now we can fully characterize imGE.

\begin{theorem}\label{thm-imGE}
Let $\pi\in S_n$ and $\ell\in\mathbb{Z}^n$.
Then there exists a braid $\beta$ such that $\pi_{\beta} = \pi, \ell_{\beta} = \ell$ iff the sum of $\ell$ has the same parity as $\pi$.
\end{theorem}
\begin{proof}
\textbf{Necessity}: Suppose such a braid exists. By Lemma \ref{lem-sum-link-writhe}, the sum of $\ell = \ell_{\beta}$ must be the writhe of $\beta$.
This must have the same parity as $\pi_{\beta} = \pi$ by Theorem \ref{thm-structure-imWP}, and we are done.

\textbf{Sufficiency}: Suppose that the sum of $\ell$ has the same parity as $\pi$.
We can construct a braid $\beta_1$ with permutation $\pi$. By Theorem \ref{thm-structure-imWP} and Lemma \ref{lem-sum-link-writhe}, $\ell_{\beta_1}$ must have a sum with the same parity as $\pi$.
By Lemma \ref{lem-all-evensum-linkvecs-possible}, there is a pure braid $\beta_2$ such that $\ell_{\beta_2} = \ell - \ell_{\beta_1}$ (since the right hand side has even sum).
Then the braid $\beta_3 := \beta_1\beta_2$ has permutation $\pi$ and vector $\ell$, as desired, and we are done.
\end{proof}

\subsection{Kernel of GE}\label{subsec-kerGE}

\hspace{\parindent}It is natural to consider the kernel of the Gauss-Epple homomorphism.
We find that imWP is a quotient group of imGE, as kerWP is a supergroup of kerGE.

Since the image of GE is not $B_n$, the kernel of GE must necessarily be nontrivial.
In fact, by computational means, we found several explicit examples:
\begin{itemize}
    \item $\sigma_1^2\sigma_2^2\sigma_1^{-2}\sigma_2^{-2}$ (which corresponds to the Whitehead link \cite{WMW-WhiteheadLink});
    \item  $\sigma_1^{-1}\sigma_3^{-1}\sigma_2^2\sigma_3^{-1}\sigma_1^{-1}\sigma_2^2$;
    \item $(\sigma_1\sigma_2^{-1})^3$;
    \item $(\sigma_2\sigma_1^{-1})^3$;
    \item $\sigma_1\sigma_2^{-1}\sigma_1^2(\sigma_1\sigma_2^{-1})^2\sigma_2^{-2}$;
    \item $(\sigma_1\sigma_2\sigma_1^2\sigma_2^{-1})^2\sigma_1^{-2}$.
\end{itemize}

\hspace{\parindent}We can make several observations about the structure of kerGE: 
We know that the kernel of the Gauss-Epple action must be a subgroup of the group of pure braids, which in turn is a subgroup of the braid group $B_n$.
Since $B_n$ has no torsion, the Gauss-Epple kernel must also have no torsion and be infinite.

\subsubsection{Random braids}

\hspace{\parindent}We study the probability that a random braid of $n$ generators lies in the kernel of the Gauss-Epple action.

Let $G$ be an arbitrary group and $H$ an arbitrary normal subgroup. 
Then, for all $g\in G$, we know that $g\in H$ iff $\phi(g) = e$, where $\phi$ is the quotient map $G\to G/H$.
(This effectively reduces our problem to studying the quotient group $K := G/H$.)
Now, defining $V(N)$ to be the number of elements of $K$ that can be produced from words of length $N$, we have the following deep theorem:
\begin{theorem}\label{thm:randwalkprob-vertcount-asy} \cite{alexopoulos1997convolution, RWIGGreview, AMSNotice200109}
For any $d$, the probability of a random walk on $K$ returning to the identity on the $N$th step is on the order of $N^{-d/2}$ iff $V(N)$ is comparable to $N^d$.
\end{theorem}

To apply this theorem, we take $G = B_n, H = \ker GE, K = \im GE$.
Since $V(N)$ is comparable to $N^n$ thanks to Theorem \ref{thm-imGE}, we conclude by Theorem \ref{thm:randwalkprob-vertcount-asy} that the probability without filtering of an element being in the kernel of Gauss-Epple is asymptotically on the order of $N^{-n/2}$, which is the result we seek.

\section{The symmetric Gauss-Epple homomorphism}\label{sec-symGE}

\hspace{\parindent}In his paper \cite{Epple98}, Epple introduced yet another action of the braid group, which he called the \textit{symmetric Gauss-Epple homomorphism}. This object is defined as follows:

\begin{definition}[Symmetric Gauss-Epple action]
For any $n\in\mathbb{N}$, the symmetric Gauss-Epple action $symGE: B_n\times \mathbb{Z}^2\to \mathbb{Z}^2$ is the unique left group action of $B_n$ (with canonical Artin generators $\sigma_1, \sigma_2, \dots$) on $\mathbb{Z}^2$ defined by the following relation:

$$\forall k\in[n], a, b\in\mathbb{Z}:
\symGE(\sigma_k, (a, b)) = \begin{cases}
(a, b) & a\notin\{k, k + 1\}\\
(k + 1, b + 1) & a = k\\
(k, b + 1) & a = k + 1
\end{cases}.$$
\end{definition}

For similar reasons as in the proof of Theorem \ref{thm-factor-imGE}, the symmetric Gauss-Epple action (which we abbreviate symGE) also factors through $\mathbb{Z}^n\rtimes S_n$ into $\Sym(\mathbb{Z}^2)$:

\begin{theorem}[Structure theorem for imGE]\label{thm-factor-imsymGE}
Define the group inclusion $\iota: \mathbb{Z}^n\rtimes S_n\to\Sym([n]\times\mathbb{Z})$ by the equation $(\pi, \ell)\to ((a, b)\to (\pi(a), b + \ell(a))$,
and define the group surjection $\varpi: B_n\to\mathbb{Z}^n\rtimes S_n$ by 
$$\varpi(\sigma_k) := (\vec{e_k} + \vec{e_{k + 1}}, (k, k + 1)).$$
Then $symGE = \iota\circ\varpi$.
\end{theorem}
\begin{proof}
Trivial.
\end{proof}

We then note that the symmetric Gauss-Epple action and the Gauss-Epple action have the same kernel:

\begin{theorem}
The kernel of the symmetric Gauss-Epple action is the same as the kernel of the ordinary Gauss-Epple action.
\end{theorem}
\begin{proof}
Suppose $\beta$ is an arbitrary pure braid. We show that $GE(\beta) = (\ell, \text{id}) \leftrightarrow \symGE(\beta) = (2\ell, \text{id})$.
It is clear that the ``permutation parts'' are all identity, so we only focus on the ``vector parts''.
It suffices to then show this claim for the the generators of the pure braid group $P_n$.
Note that $P_n$ (the pure braid group on $n$ strands) is generated by $A_{i, j} := \sigma_{j-1}\dots\sigma_{i+1}\sigma_i^2\sigma_{i+1}^{-1}\dots\sigma_{j-1}^{-1}$ \cite{Suciu16}; for example, the generators for $P_3$ are $\sigma_1^2, \sigma_2^2, \sigma_1\sigma_2^2\sigma_1^{-1}.$

We compute iteratively that $GE(A_{i, j}) = (\vec{e_i} + \vec{e_j}, \text{id}), \symGE(A_{i, j}) = (2\vec{e_i} + 2\vec{e_j}, \text{id})$. Hence the claim holds, and we are done.
\end{proof}

\begin{remark}
By the First Isomorphism Theorem, the previous theorem implies that the images of the symmetric Gauss-Epple action and the images of the regular Gauss-Epple action must be isomorphic as groups.
\end{remark}

\section{The super-Gauss-Epple homomorphism}\label{sec-superGE}

\hspace{\parindent}We define another homomorphism, this time of type signature $B_n\to \mathbb{Z}^{n(n-1)}\rtimes S_n$, which we name the \textit{super-Gauss-Epple homomorphism} (abbreviated SGE).
More precisely, define $O_{i, j}$ to be the matrix with a 1 entry at the $(i, j)$th place and 0 entries everywhere else. 
Then we define SGE as follows:
\begin{definition}[super-Gauss-Epple homomorphism]\label{def-SGE}
The super-Gauss-Epple homomorphism $SGE: B_n\to \mathbb{Z}^{n(n-1)}\rtimes S_n$ is defined by the following equation:

$$SGE(\sigma_i) := (O_{i, i + 1}, (i, i + 1)).$$
\end{definition}

(Here, by slight abuse of notation, we consider elements of $\mathbb{Z}^{n(n-1)}$ to be $n\times n$ matrices with all zeros along the diagonal.)
An example calculation of SGE is depicted in Figure \ref{fig-calc-SGE}.

\begin{figure}[h]
    \begin{center}\includegraphics[width=0.6\textwidth]{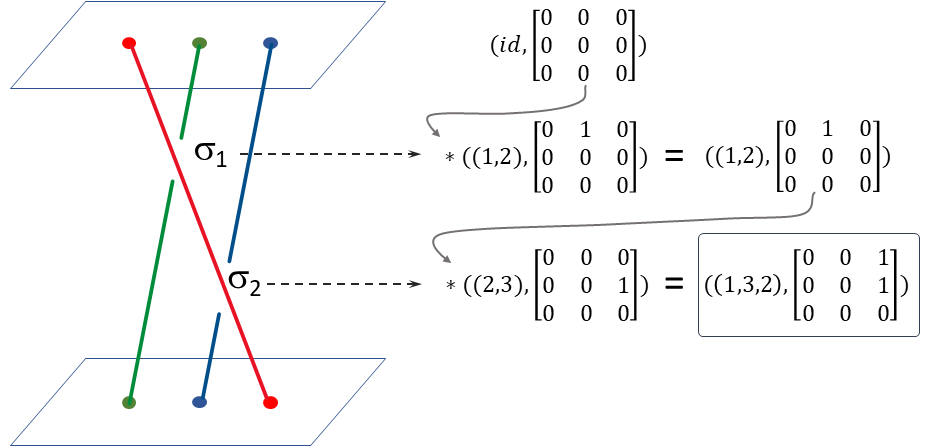}\end{center} 
    \caption{An example calculation of SGE.}
    \label{fig-calc-SGE}
\end{figure}

To verify that this is indeed a homomorphism, we verify the braid relations (as is sufficient and necessary):

\begin{align*}
    SGE(\sigma_i\sigma_{i+1}\sigma_i) &= SGE(\sigma_i)\cdot SGE(\sigma_{i+1})\cdot SGE(\sigma_i)\\
    &= (O_{i, i + 1}, (i, i + 1))\cdot(O_{i + 1, i + 2}, (i + 1, i + 2))\cdot(O_{i, i + 1}, (i, i + 1))\\
    &= (O_{i, i + 1} + (i, i+1)\cdot O_{i + 1, i + 2} + (i, i+1)(i+1, i+2)\cdot O_{i, i + 1}, \\
    &\hspace{15pt}(i, i+1)(i+1, i+2)(i, i+1))\\
    &= (O_{i, i + 1} + O_{i, i + 2} + O_{i + 1, i + 2},(i, i+1)(i+1, i+2)(i, i+1))\\
    &= (O_{i, i + 1} + O_{i, i + 2} + O_{i + 1, i + 2}, (i, i+2, i+1)),
\end{align*}
and
\begin{align*}
    SGE(\sigma_{i+1}\sigma_i\sigma_{i+1}) &= SGE(\sigma_{i+1})SGE(\sigma_i)SGE(\sigma_{i+1})\\
    &= (O_{i + 1, i + 2}, (i + 1, i + 2))\cdot(O_{i, i + 1}, (i, i + 1))\cdot(O_{i + 1, i + 2}, (i + 1, i + 2))\\
    &= (O_{i + 1, i + 2} + (i+1, i+2)\cdot O_{i, i + 1} + (i+1, i+2)(i, i+1)\cdot O_{i + 1, i + 2},\\
    &\hspace{15pt}(i+1, i+2)(i, i+1)(i+1, i+2))\\
    &= (O_{i, i + 1} + O_{i, i + 2} + O_{i + 1, i + 2},\\
    &\hspace{15pt}(i, i+1)(i+1, i+2)(i, i+1))\\
    &= (O_{i, i + 1} + O_{i, i + 2} + O_{i + 1, i + 2}, (i, i+2, i+1)),
\end{align*}
as desired.

A braid's image under the super-Gauss-Epple homomorphism stores at least as much information about it as the ordinary Gauss-Epple homomorphism. Note that this is strictly more information, as we can give the explicit example $\sigma_2^2\sigma_3^{-1}\sigma_2\sigma_1^2\sigma_2^{-1}\sigma_3^{-1}\sigma_1^{-2}$.

(By the Third Isomorphism Theorem, this also follows iff the image of the super-Gauss-Epple homomorphism is not isomorphic to that of the ordinary Gauss-Epple homomorphism.)

It is to be noted that the kernel of SGE is still nontrivial.
For example, it contains the braid $(\sigma_1\sigma_2^{-1})^3$, which was mentioned before as being a nontrivial example of the kernel of regular GE.

Hence, we have the following commutative diagram:

\[\begin{tikzcd}
	&& {B_n/[P_n, P_n]} \\
	{B_n} && {P_n^{ab}\rtimes S_n} & {\Sym(\mathbb{Z}\times\binom{n}{2})} \\
	&& {\mathbb{Z}^n\rtimes S_n} & {\Sym(\mathbb{Z}\times[n])} \\
	&& {\mathbb{Z}\times S_n}
	\arrow[from=1-3, to=2-3]
	\arrow[from=2-3, to=3-3]
	\arrow[from=3-3, to=4-3]
	\arrow["UGE", from=2-1, to=1-3]
	\arrow["SGE", from=2-1, to=2-3]
	\arrow["GE", from=2-1, to=3-3]
	\arrow["WP", from=2-1, to=4-3]
	\arrow[from=2-3, to=2-4]
	\arrow[from=3-3, to=3-4]
\end{tikzcd}\]

\subsection{A 1-cocycle}\label{section-braids-1cocy}

\hspace{\parindent}Firstly, note that the matrix of (the super-Gauss-Epple homomorphism of) a pure braid is always symmetric:

\begin{lemma}\label{lem-SGE-pure-symmetric}
Let $\beta\in P_n$, and suppose that $SGE(\beta) = (M, \text{id})$. Then $M$ is a symmetric matrix.
\end{lemma}
\begin{proof}
Recall that $P_n$ (the pure braid group on $n$ strands) is generated by 
$$A_{i, j} := \sigma_{j-1}\dots\sigma_{i+1}\sigma_i^2\sigma_{i+1}^{-1}\dots\sigma_{j-1}^{-1}$$ \cite{Suciu16}; for example, the generators for $P_3$ are $\sigma_1^2, \sigma_2^2, \sigma_1\sigma_2^2\sigma_1^{-1}.$

Note that $SGE(A_{i, j}) = (O_{i, j} + O_{j, i}, id)$, and the right component is clearly symmetric.
Since SGE matrices add for pure braids, the conclusion holds by induction.
\end{proof}

Hence, the difference between the upper half and the (transposed) lower half is always only determined by the braid permutation.

Now, let $F: B_n\to M_{n, n}(\mathbb{Z})$ be defined so that $SGE(\beta) = (M, s_{\beta}) \implies F(\beta) = M - M^T$.
Then $F(\beta) = 0\forall\beta\in P_n$.
Furthermore, $F$ satisfies the equality $F(\beta\gamma) = F(\beta) + M_{s_{\beta}}F(\gamma)M_{s_{\beta}}^{-1}$.
Hence $F(\pi\beta) = F(\beta) + id\cdot F(\beta) = F(\beta)~\forall\beta\in B_n, \pi\in P_n$, so that the value of $F(\beta)$ only depends on the permutation of the braid $\beta$, and that in general $F$ projects down to a function $\overline{F}: S_n\to M_{n, n}(\mathbb{Z})$.
Furthermore, we have $\overline{F}(\beta\gamma) = F(\beta) + M_{s_{\beta}}F(\gamma)M_{s_{\beta}}^{-1}$ (which we shall abbreviate as $\overline{F}(\beta\gamma) = \overline{F}(\beta) + \beta\cdot\overline{F}(\gamma)$. 
Therefore, $\overline{F}$ is a 1-cocycle of the action of $S_n$ on $M_{n, n}(\mathbb{Z})$, and in fact of the action on the set of antisymmetric $n\times n$ matrices over the integers.

While computing values of $\overline{F}$ for some random braids, we noticed that the entries in the upper half were always in $\{0, 1\}$. We prove this inductively as follows:

\begin{theorem}\label{thm-overlineF-upperhalf-zeroone}
Let $\varrho$ be an arbitrary permutation, and $i, j$ be integers with $i > j$. Then $\overline{F}(\varrho)[i, j]\in\{0, 1\}$.
\end{theorem}
\begin{proof}
We assume this for some $\varrho$ and prove it for $\tau\varrho$, where $\tau := (\iota, \iota + 1)$ is a transposition.

Since $\overline{F}(\tau\varrho) = \overline{F}(\tau) + \tau\cdot\overline{F}(\varrho)t$, the conjecture is instantly verified for all $i, j$ except for $\iota, \iota + 1$ (since the first term is trivially zero and the second term is in $\{0, 1\}$ by the inductive hypothesis).
In the special case $i = \iota, j = \iota + 1$, we compute \begin{align*}
\overline{F}(\tau\varrho)[\iota, \iota + 1] &= \overline{F}(\tau)[\iota, \iota + 1] + (\tau\cdot\overline{F}(\varrho)[\iota, \iota]\\
&= 1 - \overline{F}(\varrho)[\iota, \iota + 1]\\
&\in\{0, 1\}.
\end{align*}
(Here, we use the inductive hypothesis $\overline{F}(\varrho)[\iota, \iota + 1]\in\{0, 1\}$.)
Hence the claim holds.

By induction on the minimum number of simple transpositions required to be multiplied together to represent an arbitrary permutation, we are then done.
\end{proof}

This means that the image of the super-Gauss-Epple homomorphism is a group with $\mathbb{Z}^{\frac{n(n-1)}{2}}$ as a normal subgroup and $S_n$ as the quotient group.
Note that it is \textit{not} a semidirect product, since no braid with permutation $(1,2)$ has a square in the kernel of the super-Gauss-Epple homomorphism.
(Indeed, suppose such a braid $\beta = \sigma_1\varpi$ existed, and let $SGE(\beta) = (M, (1, 2))$ for some matrix $M$. Then the matrix of $SGE(\beta^2)$ is $M + (1, 2)\cdot M$. But since $M[1, 2]$ is one more than $((1, 2)\cdot M)[1, 2] = M[2, 1]$ and both are integers, the two cannot sum to zero, which is a contradiction.)

\section{Artin groups of crystallographic type}\label{sec-ATgroupsWeyl}

\hspace{\parindent}In this section, we explore generalizations of the Gauss-Epple homomorphism to the setting of Artin groups, which naturally project to corresponding Coxeter groups (generalizing how the braid group projects to the symmetric group).

\subsection{Special cases}

\subsubsection{The $I_2(4), I_2(6)$ cases}

\hspace{\parindent}This are special cases of the \textit{dihedral Artin groups}, which are the Artin groups of type $I_2(2n)$ for integers $n$.
The abelianization is $\mathbb{Z}^2$ and the Coxeter group is $D_{2n}$. Hence, the generators map to $D_{2n}\times\mathbb{Z}^2$ as follows:

\begin{center}
\begin{tabular}{|c|c|}
\hline
generator & $D_{2n}\times\mathbb{Z}^2$\\
\hline
$a$ & $(s, (1, 0))$\\
$b$ & $(sr, (0, 1))$\\
\hline
\end{tabular}
\end{center}

For GE-like purposes, we may consider $D_{2n}$ as a subset of $S_n$.

For the case $n = 4$, with $a\to (13), b\to (01)(23)$ in permutation notation, this gives rise to the following system of equations:

$$\begin{cases}
\ell_{a, 1} + \ell_{b, 2} + \ell_{a, 2} + \ell_{b, 3}  &= \ell_{b, 1} + \ell_{a, 1} + \ell_{b, 4} + \ell_{a, 4} \\
\ell_{a, 2} + \ell_{b, 1} + \ell_{a, 3} + \ell_{b, 2}  &= \ell_{b, 2} + \ell_{a, 4} + \ell_{b, 1} + \ell_{a, 3} \\
\ell_{a, 3} + \ell_{b, 4} + \ell_{a, 4} + \ell_{b, 1}  &= \ell_{b, 3} + \ell_{a, 3} + \ell_{b, 2} + \ell_{a, 2} \\
\ell_{a, 4} + \ell_{b, 3} + \ell_{a, 1} + \ell_{b, 4}  &= \ell_{b, 4} + \ell_{a, 2} + \ell_{b, 3} + \ell_{a, 1} \\
\end{cases}$$
(We used a computer program to generate this system of equations.)

Solving this system manually reveals a six-dimensional space: 
$$\ell_a = (\ell_{a, 1}, \ell_{a, 2}, \ell_{a, 3}, \ell_{a, 2}), \ell_b = (\ell_{b, 1}, \ell_{b, 2}, \ell_{b, 3}, \ell_{b, 2} + \ell_{b, 3} - \ell_{b, 1}).$$

Similarly, for the case $n = 6$ (also denoted $G_2$), we make the ansatz $a\to(26)(35), b\to(12)(36)(45)$, based on an action of the Coxeter group on the hexagon.
With the map $a\to\ell_a:= (\ell_{a, 1}, \ell_{a, 1}, \dots, \ell_{a, 6}), b\to\ell_b := (\ell_{b, 1}, \dots, \ell_{b, 6})$, this produces the following system of linear equations:

$$\begin{cases}
\ell_{a, 1} + \ell_{b, 1} + \ell_{a, 2} + \ell_{b, 6} + \ell_{a, 3} + \ell_{b, 5} &= \ell_{b, 1} + \ell_{a, 2} + \ell_{b, 6} + \ell_{a, 3} + \ell_{b, 5} + \ell_{a, 4}\\
\ell_{a, 2} + \ell_{b, 6} + \ell_{a, 3} + \ell_{b, 5} + \ell_{a, 4} + \ell_{b, 4} &= \ell_{b, 2} + \ell_{a, 1} + \ell_{b, 1} + \ell_{a, 2} + \ell_{b, 6} + \ell_{a, 3}\\
\ell_{a, 3} + \ell_{b, 5} + \ell_{a, 4} + \ell_{b, 4} + \ell_{a, 5} + \ell_{b, 3} &= \ell_{b, 3} + \ell_{a, 6} + \ell_{b, 2} + \ell_{a, 1} + \ell_{b, 1} + \ell_{a, 2}\\
\ell_{a, 4} + \ell_{b, 4} + \ell_{a, 5} + \ell_{b, 3} + \ell_{a, 6} + \ell_{b, 2} &= \ell_{b, 4} + \ell_{a, 5} + \ell_{b, 3} + \ell_{a, 6} + \ell_{b, 2} + \ell_{a, 1}\\
\ell_{a, 5} + \ell_{b, 3} + \ell_{a, 6} + \ell_{b, 2} + \ell_{a, 1} + \ell_{b, 1} &= \ell_{b, 5} + \ell_{a, 4} + \ell_{b, 4} + \ell_{a, 5} + \ell_{b, 3} + \ell_{a, 6}\\
\ell_{a, 6} + \ell_{b, 2} + \ell_{a, 1} + \ell_{b, 1} + \ell_{a, 2} + \ell_{b, 6} &= \ell_{b, 6} + \ell_{a, 3} + \ell_{b, 5} + \ell_{a, 4} + \ell_{b, 4} + \ell_{a, 5}\\
\end{cases}$$

This system is highly symmetric and many of the variables cancel. 
Therefore, we can solve it exactly, producing the following linear parametrization of the entire solution space over the integers, which is nine-dimensional: 
$$\ell_a = (a_1, a_2, a_3, a_1, a_2 + x_2, a_3 + x_2), \ell_b = (b_1, b_2, b_3, b_1 + y_1, b_2 - y_1, b_3 + y_3),$$
where $a_1, a_2, a_3, b_1, b_2, b_3, x_2, y_1, y_3$ range independently over $\mathbb{Z}$.

\subsubsection{The $B_n$ case}

\hspace{\parindent}The isomorphism between the Coxeter group of this case and $C_2\wr S_n = C_2^n\rtimes S_n$ can be tabulated in generators (where we list the generators of the Coxeter graph as $a_1, a_2, \dots, a_n$) as follows:

\begin{center}
    \begin{tabular}{|c|c|c|}
    \hline
    Coxeter & $S_n$ & $C_2^n$\\
    \hline
    $a_1$ & id & $(1, 0, \dots, 0)$\\
    $a_2$ & (1, 2) & $(0, 0, \dots, 0)$\\
    $a_3$ & (2, 3) & $(0, 0, \dots, 0)$\\
    \vdots & \vdots & \vdots\\
    $a_n$ & ($n - 1$, $n$) & $(0, 0, \dots, 0)$\\
    \hline
    \end{tabular}
\end{center}

In the classical case of the braid group $B_3$, the space of possible ``linking vector'' analogues is four-dimensional, and spanned by ``symmetric GE'' ($\sigma_1 \mapsto (1, 1, 0), \sigma_2 \mapsto (0, 1, 1)$), ``row GE'' ($\sigma_1 \mapsto (0, 0, 1), \sigma_2 \mapsto (1, 0, 0)$), ``zero-ish GE 1'' ($\sigma_1 \mapsto (1, -1, 0), \sigma_2 \mapsto 0$), and ``zero-ish GE 2'' ($\sigma_1 \mapsto 0, \sigma_2 \mapsto (0, 1, -1)$). We can derive this by noting that if we have $\sigma_1 \mapsto (a_1, b_1, c_1), \sigma_2 \mapsto (a_2, b_2, c_2)$, then the single braiding relation $\sigma_1\sigma_2\sigma_1 = \sigma_2\sigma_1\sigma_2$ gives us $c_1 = a_2$ and $a_1 + b_1 = b_2 + c_2$ as constraints after simplification; some elementary algebra then gives us the four basis vectors. 

As for the images of the corresponding homomorphisms of $B_3$ over the pure braids only, the image of symmetric GE is the set of vectors of 3 even integers which sum to 0 modulo 4, the image of row GE is just the set of vectors of 3 even integers, and the images of the two ``zero-ish'' analogs are just all zero.

As for higher braid groups $B_n$, the ``far commutativity relations'' $\sigma_i\sigma_j = \sigma_j\sigma_i$ for $|i - j|\geq 2$ basically constrain all the entries of the ``linking vector'' corresponding to $\sigma_i$ (which I’ll call $\ell_i$ by analogy) right of the $(i+1)^{\text{th}}$ to be the same within, and the entries of $\ell_i$ left of the $i^{\text{th}}$ to also be the same within. Combined with the braid relations $\sigma_i\sigma_{i+1}\sigma_i = \sigma_{i+1}\sigma_i\sigma_{i+1}$, we get $\ell_{i, i+2} = \ell_{i + 1, i}$ as well as $\ell_{i, i} + \ell_{i, i+1} = \ell_{i+1, i+1} + \ell_{i+1, i+2}$. Hence the set of possible $\ell_i$ combinations can be spanned by the following:

\begin{itemize}
\item Classical symmetric GE ($\ell_i = e_i + e_{i + 1}$)
\item Row GE for row $i_0$ ($\ell_{i_0} = \sum_{i: i > i_0 + 1} e_i, \ell_{i_0 + 1} = \sum_{i: i < i_0} e_i, \ell_i = 0$ elsewise)
\item Zero-ish GE for row $i_0$ ($\ell_{i_0} = e_{i_0} - e_{i_0 + 1}, \ell_i = 0$ elsewise)
\end{itemize}

There are also corresponding 1-cocycles. We can take linear combinations thereof to get even more 1-cocycles since $\mathbb{Z}^3$ is abelian.

\subsection{Super-Gauss-Epple analogues}

\hspace{\parindent}Since $\mathcal{C}$ acts on $\Phi$, this induces an action of $\mathcal{C}$ on $\mathbb{Z}^\Phi$.
Hence, there is a homomorphism $\mathcal{A}\to\mathbb{Z}^{\Phi}\rtimes\mathcal{C}$, defined by $a\mapsto(\alpha_a, \mathcal{C}(a))$, where $\mathcal{C}(a)$ is the image of $a$ under the quotient map $\mathcal{A}\to\mathcal{C}$; by analogy, we shall also denote it $SGE$, or $SGE_{\mathcal{A}}$ when necessary.

This is a generalization of the super-Gauss-Epple homomorphism, when $\mathcal{A} = B_n, \mathcal{C} = S_n$.
Here, the root system $\Phi$ is $\{\vec{e_i} - \vec{e_j}|1\leq i\neq j\leq n\}$, which has size $n(n-1)$, and we identify the $\vec{e_i} - \vec{e_j}$ component of $\mathbb{Z}^{\Phi}$ with the $(i, j)$th component of a matrix whose diagonal elements are all zero. 

We can show that this is indeed a homomorphism with geometric arguments.

\begin{theorem}\label{thm-SGE-welldef-Artingroups}
There is a unique homomorphism $\mathcal{A}\to\mathbb{Z}^{\Phi}\rtimes\mathcal{C}$, defined by $a\mapsto(\alpha_a, \mathcal{C}(a))$, where $\mathcal{C}(a)$ is the image of $a$ under the quotient map $\mathcal{A}\to\mathcal{C}$.
\end{theorem}
\begin{proof}
We show that the braid relations of $\mathcal{A}$ are satisfied.
To do this, we show that the $\mathbb{Z}^{\Phi}$ sides are identical.
As these values are formal linear combinations of roots (elements of $\Phi$), we also consider them as multisets, and prove them equal accordingly.
(We do not consider the $\mathcal{C}$ sides because these are trivially equal due to the fact that there is a bona-fide homomorphism $\mathcal{A} \to \mathcal{C}$.)

Let $a_1$ and $a_2$ be two generators of $\mathcal{A}$ for which there is a braid relation.
For notational simplicity, we will refer to the corresponding vectors as $\vec{v}_1 := \mathcal{C}(a_1), \vec{v}_2 := \mathcal{C}(a_2)$.
We shall also introduce the notation $*$.
For two arbitrary vectors $\vec{v}$ and $\vec{w}$, we use $\vec{v} * \vec{w}$ to denote the vector that results from reflecting $\vec{w}$ over the orthogonal hyperplane of $\vec{v}$.
We will compose it right-associatively, so that $\vec{v} * \vec{w} * \vec{x}$ shall refer to $\vec{v} * (\vec{w} * \vec{x})$.
We now perform casework depending on the length of this braid relation:

\begin{figure}[H]
    \centering
    \includegraphics[height=0.2\textwidth]{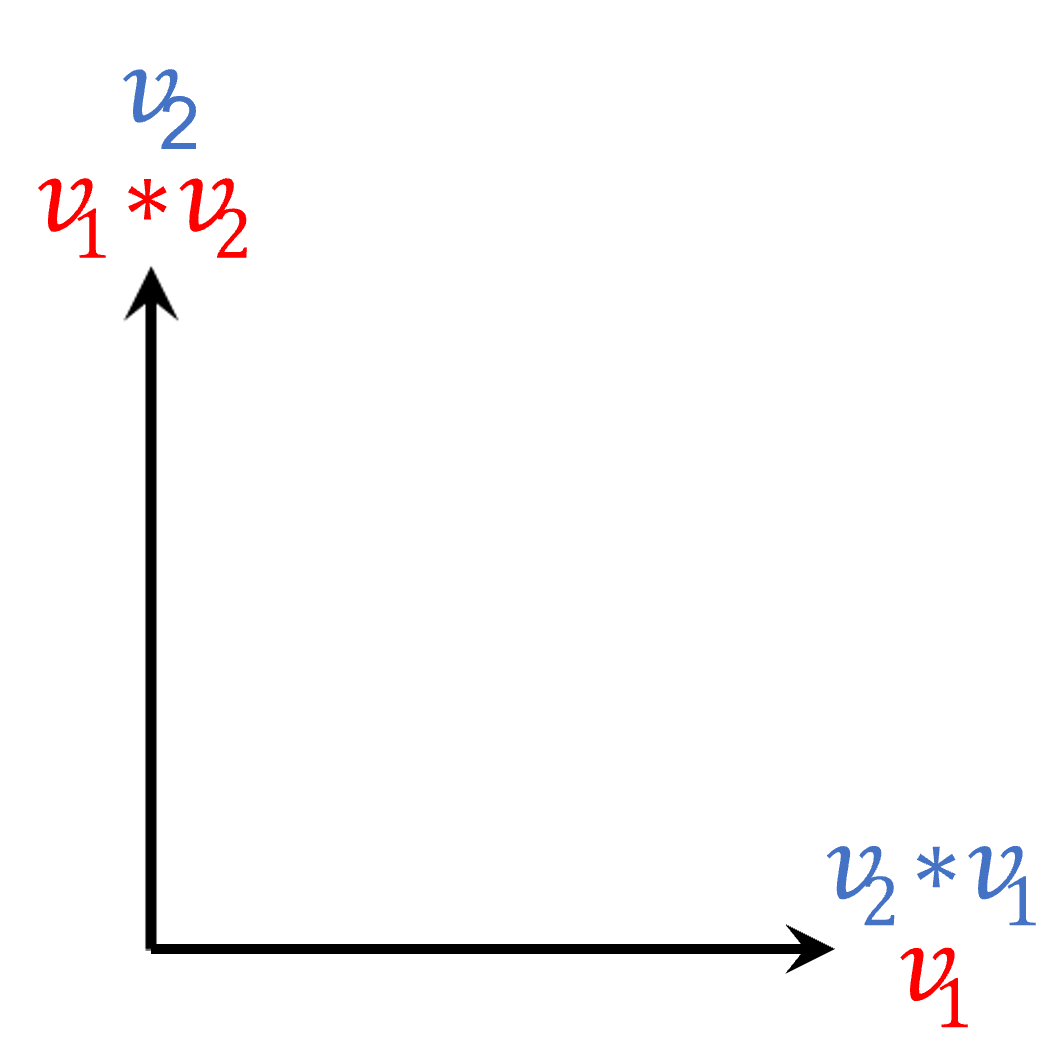}
    \caption{Two vectors for Coxeter group generators with a braid relation of length 2.}
    \label{fig:artingenrellen2}
\end{figure}
\textbf{Length 2}: We have $a_1 a_2 = a_2 a_1$, so $\vec{v}_1$ and $\vec{v}_2$ are orthogonal, as shown in Figure \ref{fig:artingenrellen2}.
We find that $\vec{v}_1 * \vec{v}_2 = \vec{v}_2$ and that $\vec{v}_2 * \vec{v}_1 = \vec{v}_1$.
Therefore, we have $\{\vec{v}_1, \vec{v}_1 * \vec{v}_2\} = \{\vec{v}_2, \vec{v}_2 * \vec{v}_1\}$, as desired.

\begin{figure}[H]
    \centering
    \includegraphics[height=0.19\textwidth]{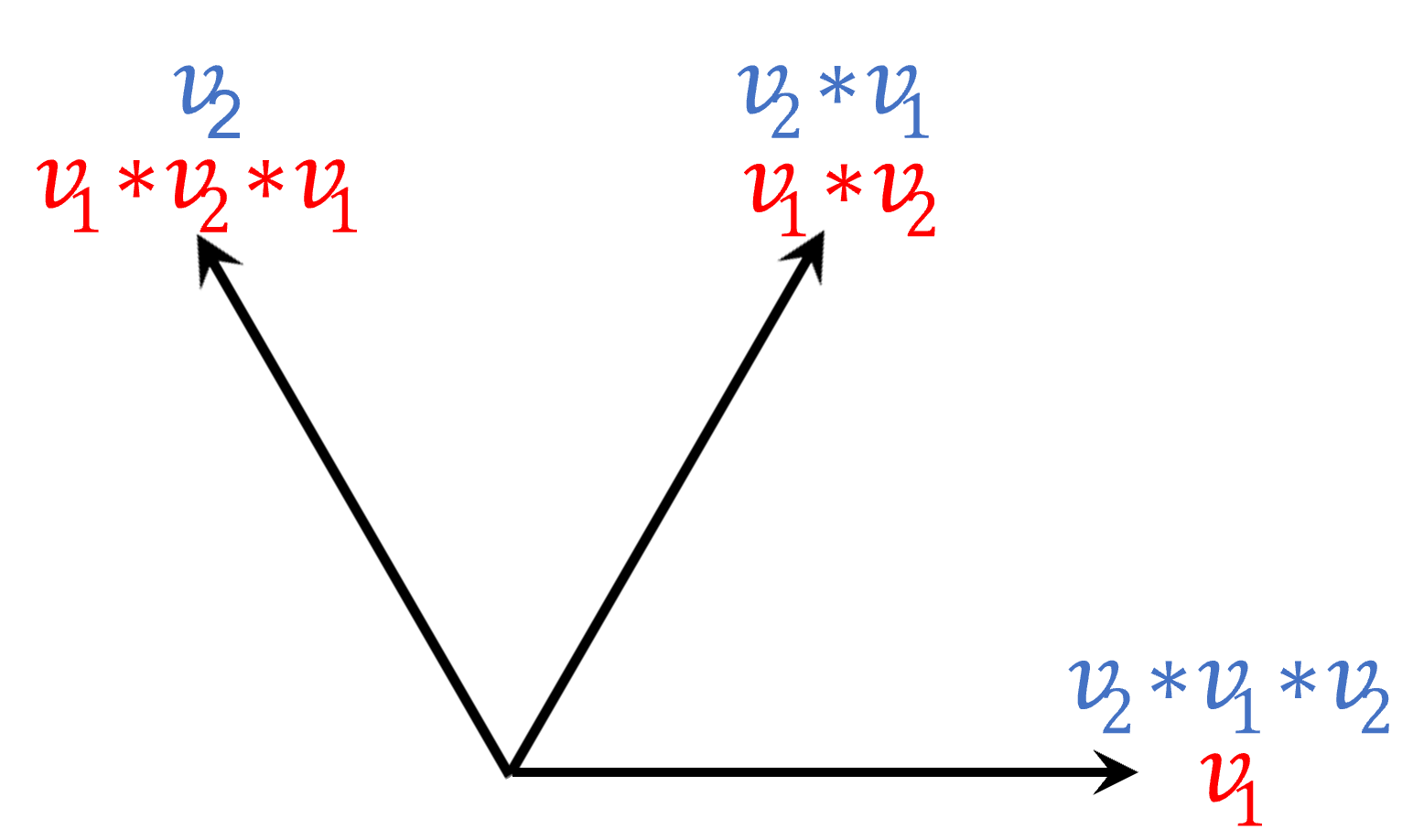}
    \caption{Two vectors for Coxeter group generators with a braid relation of length 3.}
    \label{fig:artingenrellen3}
\end{figure}
\textbf{Length 3}: We have $a_1 a_2 a_1 = a_2 a_1 a_2$, so $\vec{v}_1$ and $\vec{v}_2$ have an angle of $120^\circ$ between them, as shown in Figure \ref{fig:artingenrellen3}.
We find that $\vec{v}_1 * \vec{v}_2 = \vec{v}_2 * \vec{v}_1$ (both being the vector that bisects the angle between $\vec{v}_1$ and $\vec{v}_2$).
Furthermore, we find that $\vec{v}_1 * (\vec{v}_2 * \vec{v}_1) = \vec{v}_2$ and $\vec{v}_2 * (\vec{v}_1 * \vec{v}_2) = \vec{v}_1$.
Therefore, we have $\{\vec{v}_1, \vec{v}_1 * \vec{v}_2, \vec{v}_1 * (\vec{v}_2 * \vec{v}_1)\} = \{\vec{v}_2, \vec{v}_2 * \vec{v}_1, \vec{v}_2 * (\vec{v}_1 * \vec{v}_2)\}$, as desired.

\begin{figure}[h!]
    \centering
    \includegraphics[height=0.22\textwidth]{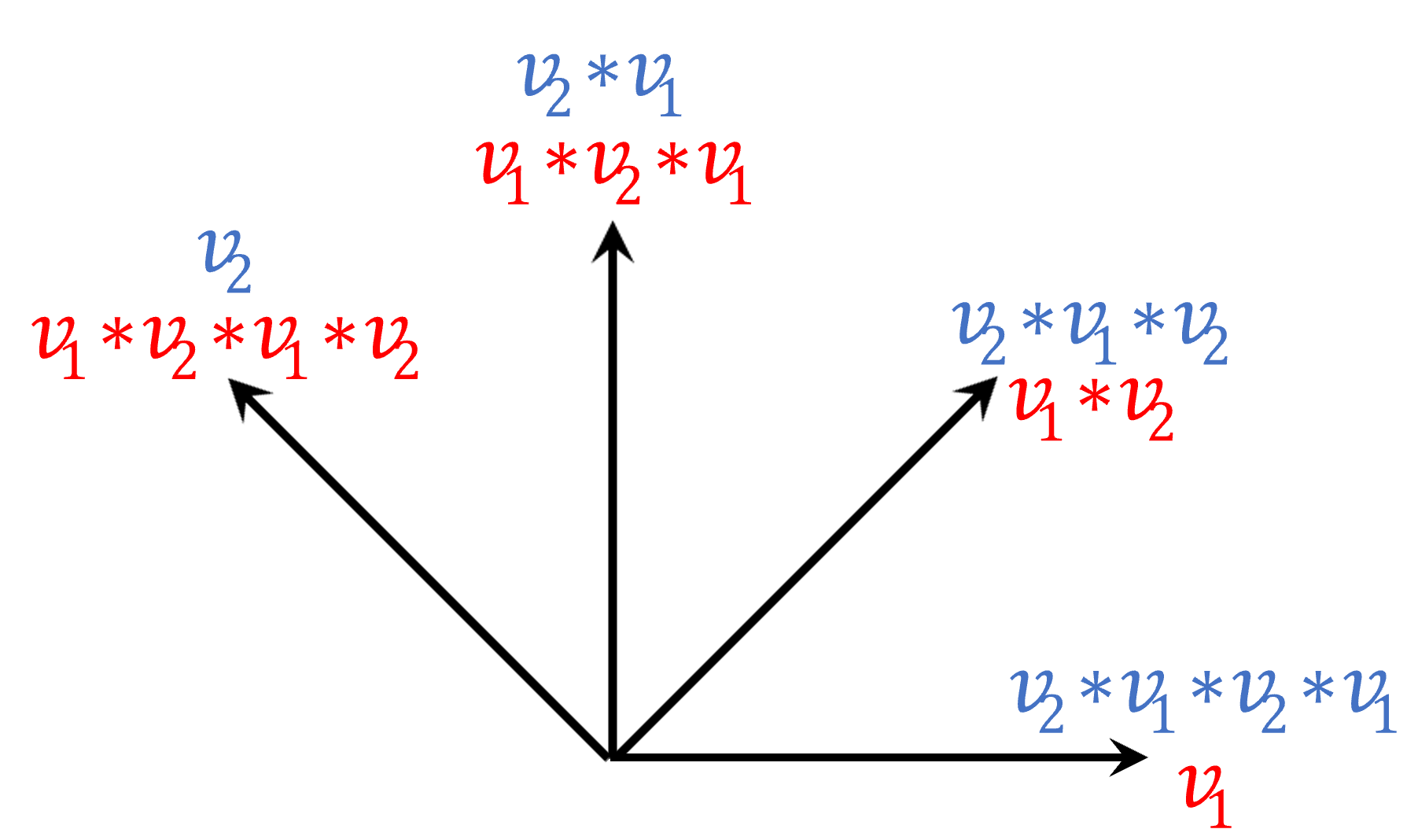}
    \caption{Two vectors for Coxeter group generators with a braid relation of length 4.}
    \label{fig:artingenrellen4}
\end{figure}
\textbf{Length 4}: We have $(a_1 a_2)^2 = (a_2 a_1)^2$, so $\vec{v}_1$ and $\vec{v}_2$ have an angle of $135^\circ$ between them, as shown in Figure \ref{fig:artingenrellen4}.
We find that $\vec{v}_1 * \vec{v}_2$ has an angle of $45^\circ$ from $\vec{v}_1$ towards $\vec{v}_2$, and is so orthogonal to $\vec{v}_2$.
Similarly, $\vec{v}_2 * \vec{v}_1$ has an angle of $45^\circ$ from $\vec{v}_2$ towards $\vec{v}_1$, and is so orthogonal to $\vec{v}_1$.
Therefore, we have $\vec{v}_1 * (\vec{v}_2 * \vec{v}_1) = \vec{v}_2 * \vec{v}_1$ and $\vec{v}_2 * (\vec{v}_1 * \vec{v}_2) = \vec{v}_1 * \vec{v}_2$.
Furthermore, we have $\vec{v}_1 * \vec{v}_2 * \vec{v}_1 * \vec{v}_2 = \vec{v}_2$ and $\vec{v}_2 * \vec{v}_1 * \vec{v}_2 * \vec{v}_1 = \vec{v}_1$.
Therefore, we conclude that $\{\vec{v}_1, \vec{v}_1 * \vec{v}_2, \vec{v}_1 * (\vec{v}_2 * \vec{v}_1), \vec{v}_1 * \vec{v}_2 * \vec{v}_1 * \vec{v}_2\} = \{\vec{v}_2, \vec{v}_2 * \vec{v}_1, \vec{v}_2 * (\vec{v}_1 * \vec{v}_2, \vec{v}_2 * \vec{v}_1 * \vec{v}_2 * \vec{v}_1\}$, as desired.

\begin{figure}[h]
    \centering
    \includegraphics[height=0.22\textwidth]{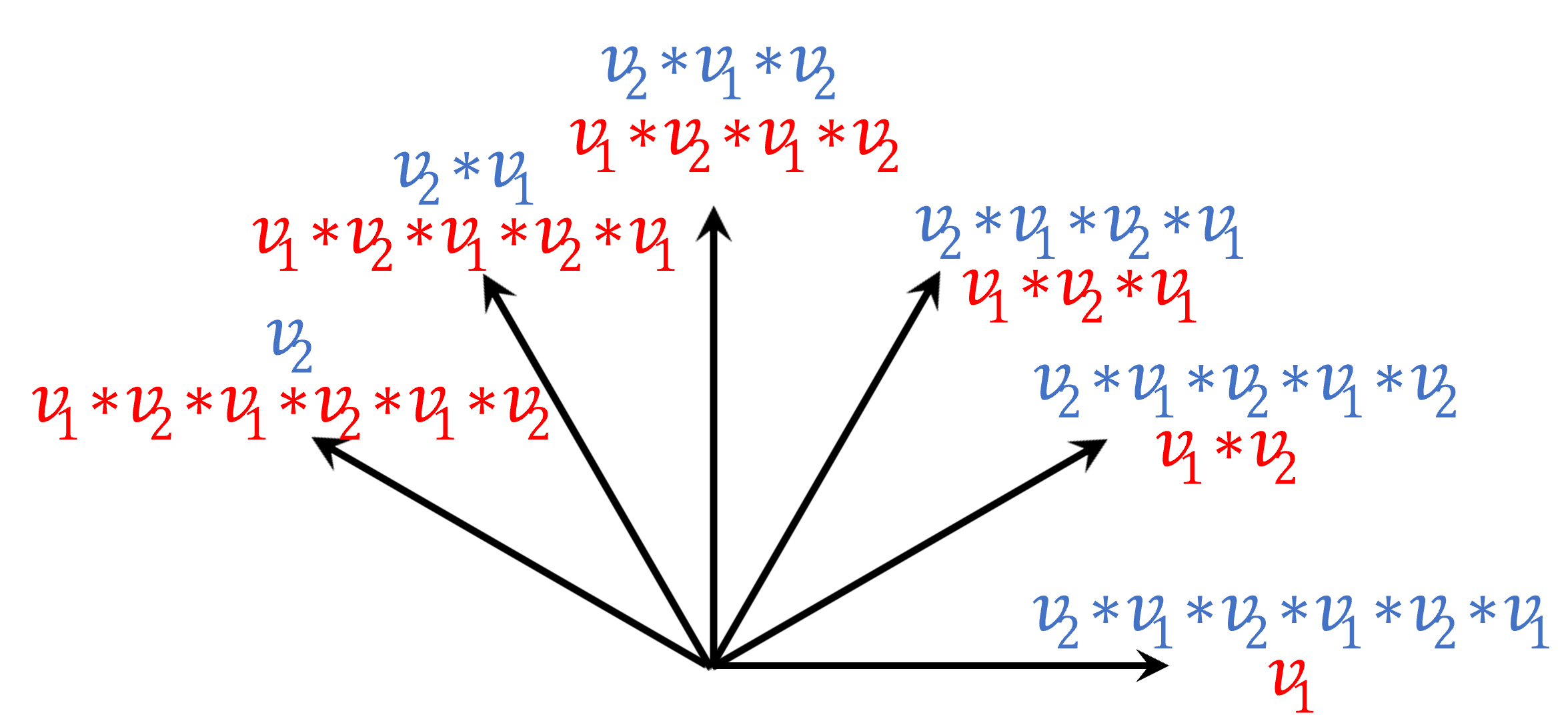}
    \caption{Two vectors for Coxeter group generators with a braid relation of length 6.}
    \label{fig:artingenrellen6}
\end{figure}
\textbf{Length 6}: We have $(a_1 a_2)^3 = (a_2 a_1)^3$, so $\vec{v}_1$ and $\vec{v}_2$ have an angle of $150^\circ$ between them, as shown in Figure \ref{fig:artingenrellen6}.
Divide this angle by four evenly spaced vectors, named $\vec{w}_1, \vec{w}_2, \vec{w}_3, \vec{w}_4$, so that the order of the vectors is $\vec{v}_1, \vec{w}_1, \vec{w}_2, \vec{w}_3, \vec{w}_4, \vec{v}_2$ (with $30$-degree gaps between each consecutive pair).
As before, we can calculate that $\vec{v}_1, \vec{v}_1 * \vec{v}_2, \vec{v}_1 * \vec{v}_2 * \vec{v}_1, \dots, \vec{v}_1 * \vec{v}_2 * \vec{v}_1 * \vec{v}_2 * \vec{v}_1 * \vec{v}_2$ is precisely $\vec{v}_1, \vec{w}_1, \vec{w}_2, \vec{w}_3, \vec{w}_4, \vec{v}_2$, in that order.
Similarly, we can compute that $\vec{v}_2, \vec{v}_2 * \vec{v}_1, \vec{v}_2 * \vec{v}_1 * \vec{v}_2, \dots, \vec{v}_2 * \vec{v}_1 * \vec{v}_2 * \vec{v}_1 * \vec{v}_2 * \vec{v}_1$ is precisely $\vec{v}_2, \vec{w}_4, \vec{w}_3, \vec{w}_2, \vec{w}_1, \vec{v}_1$, in that order.
Hence, we have $\{\vec{v}_1, \vec{v}_1 * \vec{v}_2, \vec{v}_1 * \vec{v}_2 * \vec{v}_1, \dots, \vec{v}_1 * \vec{v}_2 * \vec{v}_1 * \vec{v}_2 * \vec{v}_1 * \vec{v}_2\} = \{\vec{v}_2, \vec{v}_2 * \vec{v}_1, \vec{v}_2 * \vec{v}_1 * \vec{v}_2, \dots, \vec{v}_2 * \vec{v}_1 * \vec{v}_2 * \vec{v}_1 * \vec{v}_2 * \vec{v}_1\}$ as desired, and we are done.
\end{proof}

\subsection{Another commutative diagram}

\hspace{\parindent}We can summarize this in the following commutative diagram:

\[\begin{tikzcd}
	{\mathcal{P}} & {\mathcal{A}} & {\mathcal{C}} \\
	{\mathcal{P}/[\mathcal{P}, \mathcal{P}]} & {\mathcal{A}/[\mathcal{P}, \mathcal{P}]}
	\arrow[two heads, from=1-1, to=2-1]
	\arrow[hook', from=2-1, to=2-2]
	\arrow[hook, from=1-1, to=1-2]
	\arrow["SGE", two heads, from=1-2, to=2-2]
	\arrow[two heads, from=2-2, to=1-3]
	\arrow[two heads, from=1-2, to=1-3]
\end{tikzcd}\]

Here, every arrow is either an inclusion map (indicated as being injective) or a quotient map (indicated as being surjective).

\subsection{More 1-cocycles}

\hspace{\parindent}Again, suppose that $\mathcal{A}$ is an Artin group of finite type, $\mathcal{C}$ is its associated Coxeter group, $\Phi$ is an associated root system, and $\Delta\subset\Phi$ is an associated subset of simple roots (and let $\Delta_a$ be the associated simple root for any Artin generator $a$ of $\mathcal{A}$).
Define the function $F: \mathcal{A}\to\mathbb{Z}^{\Phi}$ by 
$$\forall a\in\mathcal{A}: SGE(a) = (\ell, \mathcal{C}(a)) \implies F(a) = \ell - \overline{\ell},$$
where $\ell$ is a vector of $\mathbb{Z}^{\Phi}$ and $\overline{\ell}$ is the image of $\ell$ under the rotation/reflection of $\mathbb{Z}^{\Phi}$ given by $e_{\alpha}\to e_{-\alpha}\forall\alpha\in\Phi$.
It is clear that $F$ satisfies the relation $F(ab) = F(a) + \mathcal{C}(a)\cdot F(b)$.

Then, we have the following lemmas (generalizing those of Section \ref{section-braids-1cocy}):

\begin{lemma}\label{lem-Fbar-exists-artingroups}
Let $F: \mathcal{A}\to\mathbb{Z}^{\Phi}$ be as defined for Artin groups of finite type, and let $p\in\mathcal{P}$ (where $\mathcal{P}$ is the kernel of the quotient map $\mathcal{A}\to\mathcal{C}$). Then $F(p) = \vec{0}$, and hence the map $F$ descends to a map $\overline{F}: \mathcal{C}\to\mathbb{Z}^{\Phi}$
\end{lemma}
\begin{proof}
Let $a$ be an Artin generator.
Then $F(a^2) = F(a) + \mathcal{C}(a)\cdot F(a) = (\phi_{\Delta_a} - \phi_{-\Delta_a}) + (\phi_{-\Delta_a} - \phi_{\Delta_a}) = \vec{0}$.

Again, let $p\in\mathcal{P}$, $\upsilon$ be an Artin generator of $\mathcal{A}$, and suppose that $F(p) = 0$.
Then $F(\upsilon p\upsilon^{-1}) = F(\upsilon) + \mathcal{C}(\upsilon)\cdot F(p) + \mathcal{C}(\upsilon p)\cdot F(\upsilon^{-1}) = (\phi_{\Delta_{\upsilon}} - \phi_{-\Delta_{\upsilon}}) + \mathcal{C}(\upsilon)\cdot\vec{0} + (\phi_{-\Delta_{\upsilon}} - \phi_{\Delta_{\upsilon}}) = \vec{0}$.

By induction, we thus show that $p\in\mathcal{P}$, $F(p) = 0$.
Hence, $F$ maps cosets of $\mathcal{P}$ in $\mathcal{A}$ into single elements of $\mathbb{Z}^{\phi}$, so we conclude $\overline{F}$ exists and we are done.
\end{proof}

\begin{lemma}\label{lem-Fbar-zeroone-artingroups}
Let $\overline{F}: \mathcal{C}\to\mathbb{Z}^{\Phi}$ be as defined in Lemma \ref{lem-Fbar-exists-artingroups}, and let $\Phi^+$ be the set of positive roots associated to the choice $\Delta$ of simple roots.
Then for all $\phi\in\Phi^+, c\in\mathcal{C}$, we have $\overline{F}(c)[\phi]\in\{0, 1\}$ (where $\overline{F}(c)[\phi]$ is the $\phi$ component of $\overline{F}(c)$).
\end{lemma}
\begin{proof}
If $c = e$, then the statement is trivial.
Hence, suppose it holds for $c$; we shall prove it for $\vartheta c$, where $\vartheta$ is a Coxeter generator of $\mathcal{C}$ (that is, the image of an Artin generator of $\mathcal{A}$).
To this end, we find, for any $\phi\in\Phi$, that $\overline{F}(\vartheta c)[\phi] = \overline{F}(\vartheta)[\phi] + (\vartheta\cdot\overline{F}(c))[\phi] = \overline{F}(\vartheta)[\phi] + \overline{F}(c)[\vartheta\star\phi]$, where $\vartheta\star\phi$ is the image of $\phi$ under the reflection across the hyperplane orthogonal to $\alpha_{\vartheta}$.
If $\alpha_{\vartheta}\neq\phi$, then the first term vanishes, and since $\vartheta\star\phi$ also belongs to $\Phi^+$, the conclusion follows.
If $\alpha_{\vartheta} = \phi$, then we have $\overline{F}(\vartheta c)[\phi] = 1 + \overline{F}(c)[-\phi] = 1 - \overline{F}(c)[\phi]\in\{0, 1\}$, and the conclusion again follows.
By induction on the length of $c$, we are done.
\end{proof}

Since this cocycle exists, the image of the super-Gauss-Epple homomorphism is not a semidirect product (as there is no way to embed the Coxeter group as a subgroup).
However, we can show that the image of $\mathcal{P}$ under the super-Gauss-Epple homomorphism is abelian.
Note that this is \textit{not} a semidirect product, for reasons analogous to those given at the end of Section \ref{section-braids-1cocy} (namely, that no element of $\mathcal{A}$ whose image in $\mathcal{C}$ is a canonical generator can have a square whose image under super-Gauss-Epple is that of the identity).

We can also characterize the image of the super-Gauss-Epple homomorphism the following way:
Note that, if $c\in\mathcal{A}$ and $a$ is a canonical generator of $\mathcal{A}$, then $SGE(ca^2c^{-1}) = (v_c, \mathcal{C}(c))(\Phi_{\Delta_a} + \Phi_{-\Delta_a}, e)(-\mathcal{C}(c)^{-1}v_c, W(c)) = (\Phi_{\mathcal{C}(c)\cdot\Delta_a} + \Phi_{\mathcal{C}(c)\cdot\Delta_a}, e)$.
Hence, the set of all pairs $(v, c)\in\mathbb{Z}^{\Phi}\rtimes C$ that belong to the image of SGE are the pairs such that $v - \overline{v} = \overline{F}(c)$.

\section{Complex reflection groups}\label{sec-complex}

We considered a candidate analogue of a ``symmetric super-Gauss-Epple homomorphism'', which would have type signature $B \to \mathbb{Z}^{\mathcal{A}}\rtimes W$: namely, a homomorphism given by generating relations $s\to (e_{H_s}, W(s))$, where $H_s$ is a distinguished generator of $B$ whose corresponding element in $W$ is a member of $\Psi$.
For example, in the case of $\mathcal{C} = G(n, 1, 1) = \mathbb{Z}/n\mathbb{Z}$, we have $\mathcal{A} = \mathbb{Z}$, and $\mathcal{HA}$ consists of a single element $\mathcal{H}$.
Therefore, the associated homomorphism is an injective homomorphism of type $\mathbb{Z}\to\mathbb{Z}/n\mathbb{Z}\times\mathbb{Z}$ given by the generating relation $1\to (1, 1)$.
It is trivial to verify that this is indeed a homomorphism.

However, using a geometric argument, we show that no such homomorphism exists, unless $B$ is an Artin group of finite type:

\begin{theorem}[No SGE for complex reflection groups]
Let $B$ be some braid group associated to a complex reflection group, $W$ be the associated complex reflection group, and $\mathcal{HA}$ the associated hyperplane arrangement.
Suppose that, in the canonical Artin-like presentation of $B$,
there is a relation of the form $aba = bab$, where $a$ and $b$ are two distinct generators.
Then there is no homomorphism $B \to \mathbb{Z}^{\mathcal{A}}\rtimes W$ given by generating relations $s\to (e_{H_s}, W(s))$, where $H_s$ is a distinguished generator of $B$ whose corresponding element in $W$ is a member of $\Psi$.
\end{theorem}
\begin{proof}
We prove $(e_{H_a}, W(a))(e_{H_b}, W(b))(e_{H_a}, W(a)) \neq (e_{H_b}, W(b))(e_{H_a}, W(a))(e_{H_b}, W(b))$.

We can easily compute that the former part of the left hand side  will be $e_{H_a} + e_{W(a)H_b} + e_{W(a)W(b)e_{H_a}}$,
and the former part of the right hand side would be $e_{H_b} + e_{W(b)H_a} + e_{W(b)W(a)e_{H_b}}$.
We claim that $H_a = W(b)W(a)H_b, H_b = W(a)W(b)H_a$, but that $W(a)H_b \neq W(b)H_a$.

Firstly, however, we shall need to use the determinant trick.
We have $W(a)W(b)W(a) = W(b)W(a)W(b)$.
Since both sides are linear transformations, we can take determinants: $\det(W(a))\det(W(b))\det(W(a)) = \det(W(b))\det(W(a))\det(W(b))$.
Since these determinants are nonzero real numbers, we can cancel to obtain $\det(W(a)) = \det(W(b))$.
This equality is important because letting $\zeta_{W(a)}$ be the multiplier of $W(a)$ and similarly for $\zeta_{W(b)}$, we can conclude that $\zeta_{W(a)} = \zeta_{W(b)}$, and will refer to both constants as merely $\zeta$.

Now, to show $H_a = W(b)W(a)H_b$, let $s_{\zeta, H}$ be the pseudo-reflection that multiplies by $\zeta$ and preserves $H$. 
We compute
\begin{align*}
    s_{W(b)W(a)H_b, \zeta} &= W(b)W(a)s_{H_b, \zeta}(W(b)W(a))^{-1}\\
    &= W(b)W(a)W(b)(W(b)W(a))^{-1}\\
    &= W(a)W(b)W(a)(W(b)W(a))^{-1}\\
    &= W(a)\\
    &= s_{H_a, \zeta},
\end{align*}
which gives the claim.
By symmetry, we can also obtain $H_b = W(a)W(b)H_a$.

Now to show that $W(a)H_b \neq W(b)H_a$, we compute

\begin{align*}
    s_{W(a)H_b, \zeta} &= W(a)s_{H_b, \zeta}W(a)^{-1}\\
    &= W(a)W(b)W(a)^{-1}\\
    &= (W(a)W(b)W(a))W(a)^{-2},
\end{align*}
and similarly $s_{W(b)H_a, \zeta} = (W(b)W(a)W(b))W(b)^{-2} = (W(a)W(b)W(a))W(b)^{-2}$.
Observe now that $W(a)^2 = s_{H_a, \zeta^2}$ and $W(b)^2 = s_{H_b, \zeta^2}$. 
Since $a$ and $b$ are distinct, this means that $H_a$ and $H_b$ are distinct.
This in turn implies that $s_{H_a, \zeta^2}$ and $s_{H_b, \zeta^2}$ are distinct unless $\zeta^2 = 1$.
However, if $\zeta^2 = 1$, then $\zeta = -1$, so that $a$ and $b$ are merely real reflections, as desired.
\end{proof}

\bibliographystyle{plain}
\bibliography{arxiv-paper}

\end{document}